\documentclass[12pt,twoside]{amsart} \usepackage{amssymb,color,tikz}
\usetikzlibrary{matrix} \nonstopmode \textwidth=16.00cm

\usepackage{xcolor}

\usepackage{amsthm, amsfonts, amssymb, amsmath, latexsym, enumerate,array,epsfig}
\usepackage{amscd}
\usepackage[all]{xy}
\usepackage{graphicx,color}
\usepackage{pgf,tikz}
\usetikzlibrary{arrows}
\graphicspath{{.}}
   \makeatletter
   \let\accent@spacefactor\relax
   \makeatother
   \usepackage[english]{babel}
\usepackage{amsfonts}

\numberwithin{equation}{section} 

\newcommand{\NN}{\mathbb{N}}
\newcommand{\kk}{{\bf k}}
\newcommand {\PP}{\mathbb{P}}

\newcommand{\cT}{\mathcal{T}}
\newcommand{\cA}{\mathcal{A}}

\newcommand{\cP}{\mathcal{P}}

\newcommand{\cO}{\mathcal{O}}

\def\cOP2{{\mathcal O}_{{\mathbb P}^2}}

 \DeclareMathOperator{\Proj}{Proj}
 \def\cocoa{{\hbox{\rm C\kern-.13em
      o\kern-.07em C\kern-.13em o\kern-.15em A}}}
\newtheorem{theorem}{Theorem}[section]
\newtheorem{lemma}[theorem]{Lemma}
\newtheorem{proposition}[theorem]{Proposition}

 \theoremstyle{definition}
\newtheorem{definition}[theorem]{Definition} \theoremstyle{remark}
\newtheorem{remark}[theorem]{Remark}
\newtheorem{example}[theorem]{Example}

\definecolor{MyDarkGreen}{cmyk}{0.7,0,1,0}

\begin{document}

\title[Saito's theorem revisited and pencils of hypersurfaces]
{Saito's theorem revisited and application   to free pencils of hypersurfaces}

\author[R. Di Gennaro]{Roberta Di Gennaro}
\address{Di Gennaro: Dipartimento di Matematica e Applicazioni “Renato Caccioppoli” \\
 Universit\`a degli Studi di Napoli Federico II\\
  80126 Napoli, Italy}
\email{digennar@unina.it}

\author[R.M. Mir\'o-Roig]{Rosa Maria Mir\'o-Roig}
\thanks{Mir\'o-Roig was partially supported by the grant PID2020-113674GB-I00}
\address{Mir\'o-Roig: Department de Mathem\`atiques i Inform\`atica\\
  Universitat de Barcelona\\
 08007 Barcelona, Spain}
\email{miro@ub.edu, ORCID 0000-0003-1375-6547}



\begin{abstract}  A hypersurface $X\subset \PP^n$  is said to be free if its associated sheaf $T_X$  of vector fields tangent to $X$  is a free ${\mathcal O}_{\PP^n}$-module. So far few examples of free hypersurfaces are known. In this short note, we reinterpret Saito's criterion of freeness in terms of multiple eigenschemes (ME) and as application we construct huge families of new examples of  free reduced hypersurfaces in $\PP^n$. All of them are union of hypersurfaces in a suitable pencil.
\end{abstract}

\maketitle

\section{Introduction} 
In this paper, we investigate the freeness of reduced hypersufaces in $\PP^n$ with $n\geq 3$. The study of freeness of divisors is a classical topic introduced by Saito for reduced divisors \cite{S} and studied by Terao for hyperplane arrangements \cite{T}. 
A reduced hypersurface $S=V(f)$ in $\PP^n$ is said to be free if its
module of logarithmic derivations is free, i.e., the direct sum of line bundles. This notion turns out to be very interesting, as
free hypersurfaces tend to be quite special and often exhibit unique properties.
Actually, few examples of free hypersurfaces are known. In the plane there is a quite extended literature on free curves, but for $n>2$ there are few examples. We recall here \cite{BC} (where binomial free divisors are characterized and some other examples are given); \cite{DS15} (where isolated cases and countable families of free (rational) surfaces in $\mathbb P^3$ and some examples in $\mathbb P^n, n \geq 3$ related to the discriminant of binary forms are given) and \cite{DS17} (where some examples of free plane arrangements and surfaces are given). 

 Our idea comes from \cite{DIMSV} where the authors use eigenschemes and pencils of curves, combined with an interpretation
of Saito’s criterion. Here we extend this approach to $\mathbb P^n, n\geq 3$ by defining Multiple Eigenschemes of closed subschemes and relating them to freeness of hypersurfaces. Finally, we give some applications.

\vskip 2mm
\noindent {\bf Acknowledgment.} The authors would like to thank J. Vall\` es for useful discussions on this subject. Indeed, our results are inspired in \cite{V} and \cite{DIMSV}.

\section{Preliminaries}

This section contains the basic definitions and results on jacobian ideals associated to reduced singular hypersurfaces in $\PP^n$ as well as on the module of logarithmic derivations. Moreover, we recall the notion of eigenscheme and we introduce ME-schemes (Multiple Eigenschemes) as a natural generalization of eigenpoints. This section lays the groundwork for the results in the later sections.

\vskip 2mm
\noindent {\bf Notation.}  From now on, we fix the polynomial ring $R=\kk$$[x_0,\cdots ,x_n]=\oplus_{d}R_d$ over an algebraically closed field $\kk $ of characteristic zero and we denote by $S: \ f = 0$ a reduced hypersurface of degree $d$ in the
$n$-dimensional projective space $\PP^n=\Proj(R)$ defined by an homogeneous polynomial $f\in R_d$. As usual, we denote by $\partial _{x_i}$, $i=0,\cdots ,n$ the partial derivatives with respect to $x_i$. Let $Der(R)=\{ a_0\partial _{x_0} +a_1 \partial _{x_1}+ \cdots +a_n\partial _{x_n} \mid a_i\in R \}$ be the free $R$-module of $\kk $-derivations of $R$.

\vskip 2mm

\subsection{The module of logarithmic derivations} As usual, for any reduced hypersurface $S : \ f = 0$ defined by a homogeneous polynomial $f$ of degree $d$, we define the {\em module of tangent derivations} (also called {\em the  module of logarithmic derivations}) $Der(f)$
as the graded $R$-module
 $$Der(f) =\{\delta \in Der(R) \mid \delta  (f)\in \langle f\rangle \}.$$
 The so-called Euler derivation $\delta _E=x_0\partial _{x_0}+\cdots +x_n\partial _{x_n}$ belongs to $Der(f)$ and we have the factorization
 $$Der(f)=R\delta _E \oplus Der_0(f)
 $$
 where $Der_0(f)=\{ \delta \in Der(f)\mid \delta(f)=0 \}$. Let $\nabla _f=(\partial _{x_0}f,\partial _{x_1}f,\cdots ,\partial _{x_n}f)$ be the vector of partial derivatives and $J_f$ the Jacobian ideal (generated by these partial derivatives); i.e. $J_f=\langle \partial _{x_0}f,\partial _{x_1}f,\cdots ,\partial _{x_n}f \rangle \subset R$.
 \\ Then $Der_0(f)$ is isomorphic to the kernel of the Jacobian map
 $$\nabla _f:R^{n+1}\longrightarrow R(d-1)$$
 i.e. $Der_0(f)$ is identified with the $R$-module of all Jacobian relations for $f$
, i.e.,
$$
Der_0(f)\cong Syz(J_f):=\left\{(a_0,a_1,\cdots ,a_n)\in R ^{n+1}\mid a_0\frac{\partial f}{\partial x_0} + \cdots +a_n\frac{\partial f}{\partial x_n} = 0 \right\}.
$$
We will denote by $Syz(J_f)_t$ the homogeneous part of degree $t$ of the graded $R$-module $Syz(J_f)$; $Syz(J_f)_t$ is a $\kk$-vector space of finite dimension.
The minimal degree of a Jacobian syzygy for $f$ is the integer $mdr(f)$ defined to be the smallest integer $r$ such that there
is a nontrivial relation
$  a_0\frac{\partial f}{\partial x_0} + a_1\frac{\partial f}{\partial x_1} + \cdots +a_n\frac{\partial f}{\partial x_n} = 0 $
among the partial derivatives  $\frac{\partial f}{\partial x_i}$, $i=0,\cdots n$,   of $f$ with coefficients $a_i \in R_r$. Therefore, we have:
$$mdr(f)=min\{ n\in \NN \mid Syz(J_f)_n\ne 0\}.
$$
It is well known that  $mdr(f) = 0$ (i.e. the $n+1$ partial  derivatives
$\frac{\partial f}{\partial x_i}$, $i=0,\cdots n$ are linearly dependent) if and only if  $S$ is
a union of hyperplanes passing through one point $p\in \PP^n$.

\begin{definition}  Let $S: \ f = 0$ be a reduced hypersurface of degree $d$ in $\PP^n$. We say that $S$ is {\em free} if the graded $R$-module $Der_0(f)$ of all Jacobian relations for $f$ is a free $R$-module, i.e., $Der_0(f)=R(-d_1)\oplus R(-d_2)\oplus\cdots \oplus  R(-d_n)$ with $d_1+d_2+\cdots +d_n=d-1$. In this case $S$ is said to be free with exponents $(d_1,d_2,\cdots ,d_n)$.
\end{definition}

Actually, few examples of free hypersurfaces are known. Below we give some of them
\begin{example}
\par\vspace{2mm}
\begin{enumerate}[(1)]
\item  The surface $S:\  f=x^6z+y^7+x^5yt+x^4y^3=0$ is free with exponents $(1,2,3)$ \cite{DS15}. Indeed, the Jacobian ideal $J_f=\langle 6x^5z+5x^4yt+4x^3y^3, 7y^6+x^5t+3x^4y^2,x^6,x^5y \rangle $ of $f$ has a resolution of the following type:
$$
0 \longrightarrow R(-3)\oplus R(-2)\oplus R(-1) \longrightarrow R^4\longrightarrow J_f(6)\longrightarrow 0.
$$
\vskip 2mm

\item For any $d\ge 10$, the surface $S:\  f=x^{d-1}z+y^d+x^{d-2}yt+x^{d-5}y^5=0$ is free with exponents $(1,4,d-6)$ \cite{DS15}. Indeed, the Jacobian ideal $J_f$ of $f$ has a resolution of the following type:
$$
0 \longrightarrow R(-d+6)\oplus R(-4)\oplus R(-1) \longrightarrow R^4\longrightarrow J_f(d-1)\longrightarrow 0.
$$
For $d=9$, the surface $S:\  f=x^{d-1}z+y^d+x^{d-2}yt+x^{d-5}y^5=0$ is not  free. Indeed, its Jacobian ideal $J_f$ has a minimal free $R$-resolution of the following type:
$$
0 \longrightarrow R(-5) \longrightarrow  R(-4)^3\oplus R(-1) \longrightarrow R^4\longrightarrow J_f(8)\longrightarrow 0.
$$

\vskip 2mm
\item The surface $S:\  f=y^2z^2-4xz^3-4y^3t+18xyzt-27x^2t^2=0$ is free with exponents $(1,1,1)$ \cite{DS15}. Indeed, the Jacobian ideal $J_f$ of $f$ has a resolution of the following type:
$$
0 \longrightarrow  R(-1)^3 \longrightarrow R^4\longrightarrow J_f(3)\longrightarrow 0.
$$

\item Starting from the triangle $T:xyz=0$, the curve $C:f=xyzc_1c_2=0$ with $c_1,c_2$ two general conics through the vertices of $T$ is free of exponent (2,4) (see also \cite[Example 2.1]{ST}).
\begin{center}
  \vspace{.5cm}  
\begin{tikzpicture}
\draw [domain=-0.5:3] plot(\x,0);
\draw [domain=-0.2:1.1] plot(\x,{(2.3*\x)});
\draw [domain=0.4:2.8] plot(\x,{(-1.5*\x+3.6)});
\draw[rotate=0] (1.2, 1) ellipse (2.3cm and 1.2cm);
\draw[rotate=55] (1.2, 0) ellipse (1.2cm and 2cm);
\end{tikzpicture}
\vspace{.5cm}
\end{center}

Indeed, its Jacobian ideal $J_f$ has a resolution of the following type:
$$
0  \longrightarrow  R(-4)\oplus R(-2) \longrightarrow R^3\longrightarrow J_f(6)\longrightarrow 0.
$$
Actually, in \cite[Proposition 4.1]{DIMSV} it is  proved that by adding conics $c_i=a_ic_1+b_ic_2$ in the pencil of $c_1,c_2$, the freeness is preserved (with exponent $(2,2k)$ if we add $k-2$ conics, but by adding $c_3,\ldots,c_{k}$ general conic through the vertices of $T$ ($k\geq 5)$, we loose the freeness. Indeed in the pencil case with $f=xyzc_1c_2\Pi_{i=1,2,3}(a_ic_1+b_ic_2)$ the Jacobian ideal $J_f$ has a minimal free $R$-resolution of the following type:
$$
0  \longrightarrow  R(-10)\oplus R(-2) \longrightarrow R^3\longrightarrow J_f(12)\longrightarrow 0.
$$
in the net case $f=xyzc_1c_2c_3c_4c_5$  the Jacobian ideal $J_f$ has a minimal free $R$-resolution of the following type:
$$
0  \longrightarrow R^2(-8) \longrightarrow  R^4(-6) \longrightarrow R^3\longrightarrow J_f(12)\longrightarrow 0.
$$

\item The following example shows how it works the classical and useful tool \lq\lq addition-deletion\rq\rq \cite{T}. \\ The plane arrangement $xyzt\prod_{i=1}^s(a_ix+b_iy)\prod _{j=1}^r(c_jz+d_jt)$ is free with exponents $(1,s+1,r+1)$. 
\\ In fact starting from the arrangement $xyzt$ that is free with exponents $(1,1,1)$, if we restrict to a plane $ax+by$ we get a triangle of lines that is free with exponent $(1,1)$, so $xyzt(ax+by)$ is free with exponents $(1,1,2)$. Inductively, adding planes of the same shape $a_ix+b_iy$, the restricted line arrangement is always the triangle and so after $s$ steps $\cA_s=xyzt\Pi_{i=1}^s (a_ix+b_iy)$ is free with exponents $(1,1, 1+s)$. Now we have to add planes $c_iz+d_it$ to $\cA_s$. The restriction of $\cA_s$ to $c_iz+d_it$ is the line arrangement consisting of the line $z=0=t$ and $s+2$ lines through a point  meeting the first line in $s+2$ distinct points. This line arrangement is free with exponents $(1,s+1)$. So the union of $\cA_s$ with the plane is free $(1,s+1,2)$ and inductively $xyzt\prod_{i=1}^s(a_ix+b_iy)\prod _{j=1}^r(c_jz+d_jt)$   is free with exponents $(1,s+1,r+1)$. 

\begin{center}
\begin{figure}
\begin{tikzpicture}[y ={(1 cm ,0 cm )} , x ={(-0.5 cm ,-0.5 cm )} ,z ={(0 cm ,1 cm )}] 
 \coordinate (O) at (0,0,0);
 \draw [-latex ] 
 (O) -- (3.5,0,0) node [ left ] {$ x $}
 ;
\draw 
[-latex ] 
(O) -- (0,3.5,0) node [ right ] {$ y $}
;
  \draw [-latex ] (O) -- (0,0,3.5) node [ above ] {$ z $};
\draw (O) -- (0,3,0) -- (3,3,0)-- (3,0,0) -- cycle;
\draw (O) -- (0,0,3) -- (3,0,3)-- (3,0,0) -- cycle;
\draw (O) -- (0,3,0) -- (0,3,3)-- (0,0,3) -- cycle;
\draw (O) -- (3,2.7,0) -- (3,2.7,3)-- (0,0,3) -- cycle;
\draw  (3,0,0) -- (0,0,3)-- (0,3,0) -- cycle;
\draw (1.5,1.4,0) -- (0,0,3) -- cycle;
\filldraw[fill=gray,line width=2pt] (0,0,0) -- (0,0,3) -- (1.5,1.4,0) -- cycle;
 \end{tikzpicture}  
 \\
 \caption{The restricted line arrangement in the plane $x-y=0$}
 \end{figure}
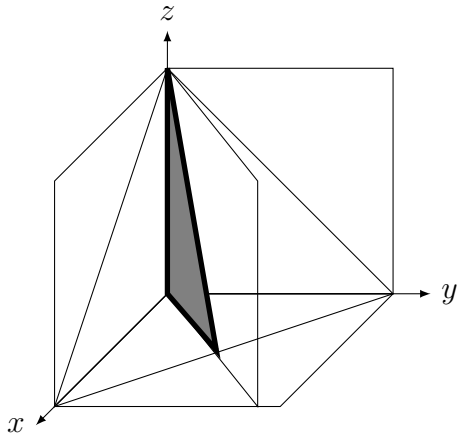
\end{center}

\end{enumerate}\end{example}
\begin{remark}Other known examples are pencils of curves in the plane, whose study is related to the knowledge of a \lq\lq canonical\rq\rq\ derivation introduced by Vall\` es in \cite{V}.
In order to generalize it in higher dimension, we consider a family generated by $n$ hypersurfaces $f_1,\ldots, f_n$ in $\mathbb P^n$ and the derivation defined as 
$$
\delta_{f_1,\ldots,f_n}=\left\vert \begin{array}{cccc}  
\partial_{x_0} & \partial_{x_1} & \cdots & \partial_{x_n} \\
\partial_{x_0} f_1 & \partial_{x_1}f_1 & \cdots & \partial_{x_n}f_1\\
\cdots &\cdots & &\\
\partial_{x_0} f_n & \partial_{x_1}f_n & \cdots & \partial_{x_n}f_n
\end{array} \right\vert.
$$
Of course $\delta_{f_1,\ldots,f_n}(f_i)=0$ for $i=1,\ldots,n$ so $\delta_{f_1,\ldots,f_n}\in Der_0(f)$ for each $f\in\langle f_1,\ldots,f_n\rangle$. 
\\ It is obvious that $\delta_{f_1,\ldots,f_n}=0$ if $\nabla_{f_i}$ are linearly dependent. For example, in $\mathbb P^3$, $\delta_{x,y,x+y}=0$, but the coordinate planes $x,y,z$ give the derivation $\delta_{x,y,z}=\partial_t\neq 0$. So, it is natural to look for geometric conditions in order to assure that $\delta_{f_1,\ldots,f_n}\neq 0$.
We note that the dimension of the base locus of the family defined by $f,g,h$ is not a strong condition in this sense. For example, letting $f=x^2-yz,g=x^2-yt,h=x^2-zt$, we get $\delta_{f,g,h}=yzt\partial_x+xy(-y+z+t)\partial_y+xz(y-z+t)\partial_z+xt(y+z-t)\partial_t$ the base locus is a zero dimensional scheme of degree $8$; instead letting $f=x^2-yz,g=y^2-xz,h=z^2-xy$, we get $\delta_{f,g,h}=(x^3+y^3+z^3-3xyz)\partial_t$ and the base locus is a cubic curve.  
\end{remark}

\vskip 2mm
It is a longstanding problem in singularities to determine whether a hypersurface is free. 
A hypersurface  $S: \ f=0$ in $\PP^n$ is free if
 the Jacobian ideal $J_f$ of $f$ is saturated and defines an arithmetically Cohen-Macaulay subscheme of codimension two.  Cohen-Macaulay ideals of codimension 2 are completely described by the so-called 
Hilbert-Burch theorem \cite{e}: if $I = \langle g_1,\ldots,
g_m\rangle \subset R$ is a Cohen-Macaulay ideal of codimension two, then $I$ is defined
by the maximal minors of the $(m+1)\times m$ matrix of the first
syzygies of the ideal $I$. As Saito's criterion says (see also Section 3) combining this with Euler's formula for a homogeneous
polynomial we obtain that a free hypersurface $S: \ f=0$ in $\PP^n$ has a very
constrained structure: $f = \det(M)$ for a $(n+1) \times (n+1)$ matrix $M$, with one row consisting of the $n+1$ variables, and the
remaining $n$ rows the minimal first syzygies on $J_f$.


\subsection{Eigenschemes and ME-schemes in $\PP^n$}

There are different notions of eigenvectors and eigenvalues for tensors as introduced independently in \cite{L} and \cite{Q}. Here we focus our attention on the algebraic-geometric point of view and we also introduce the notion of multiple eigenscheme (ME-scheme, for short) as a natural generalization of the concept of eigenscheme. To this end, we choose a basis for $\kk$$^{n+1}$ and we identify a partial symmetric tensor $T$  with a $(n+1)$-uple of homogeneous polynomials of degree $d-1$. We describe the eigenpoint of a tensor $T$ (respectively, the ME-scheme of $r$ tensors $T_1, \cdots ,T_r$) algebraically by the vanishing of the minors of a homogeneous matrix $M_T$ (respectively, $M_{T_1,\cdots ,T_r}$). More precisely, we have:

\begin{definition}  \label{def: eigenscheme}
Let $T_i=(g_0^{i},g_1^{i}, \cdots ,g_n^{i})\in (Sym^{d_i-1}\kk$$^{n+1})^{\oplus (n+1)}$, $i=1,\cdots ,r$, be $r$ partially symmetric tensors. The {\em eigenscheme} of $T_i$ is the closed subscheme $E(T_i)\subset \PP^{n+1}$ defined by the $2\times 2$ minors of the homogeneous matrix
\begin{equation}\label{def_matrix_2}
M_{T_i}=\begin{pmatrix} x_0 & x_1 & \cdots   & x_n \\
g_0^{i} & g_1^{i}  & \cdots  & g_n^{i}
\end{pmatrix}, \ i=1,\cdots r.
\end{equation}

The {\em Multiple Eigenscheme} (ME-scheme for short)  of $T_1, \cdots , T_r$ is the closed subscheme $E(T_1,\cdots ,T_r)\subset \PP^{n+1}$ defined by the $(r+1)\times (r+1)$ minors of the homogeneous matrix
\begin{equation}\label{def_matrix}
M_{T_1,\cdots, T_r}=\begin{pmatrix} x_0 & x_1  & \cdots  & x_n \\
g_0^{1} & g_1^{1}  & \cdots  & g_n^{1}
\\
\vdots & \vdots & \cdots & \vdots \\
g_0^{r} & g_1^{r}  & \cdots  & g_n^{r}
\end{pmatrix}.
\end{equation}

 If $T_1, \cdots T_r$ are general, then $E(T_i)$ is a $0$-dimensional scheme (see, for instance, \cite{Abo}) and $E(T_1,\cdots, T_r)$ is a $(r-1)$-dimensional scheme.
Moreover, by the Hochster-Eagon Theorem \cite{HE}, $R/I(E(T_i))$ and $R/I(E(T_1,\cdots ,T_r))$  are Cohen-Macaulay rings and, hence, the ideals $I(E(T_i))$ and $I(E(T_1,\cdots ,T_r))$ are  saturated. Therefore, $E(T_i)$ and $E(T_1,\cdots ,T_r)$ are standard determinantal schemes.
When the tensor $T_i$ is symmetric, i.e., there is a homogeneous polynomial $f_i$ such that $g_j^{i}=\frac{\partial f_i}{\partial x_j}$ for $j=0, \cdots ,n$, we denote its eigenscheme by $E(f_i)$ and the corresponding ME-scheme by $E(f_1,\cdots ,f_r)$.
\end{definition}

Set theoretically $E(T_i)$ is the union of the base locus and the fixed points of the rational map
$$\PP^n \cdots \longrightarrow \PP^n $$  $$ \hspace{3cm} x\mapsto g^{i}(x):=(g_0^{i}(x),g_1^{i}(x),\cdots ,g_n^{i}(x)).$$
In particular, the eigenpoints $E(f_i)$ are the base locus and the fixed points of
 polar map  $$\nabla f_i=(\partial_{x_0} f_i,\partial_{x_1} f_i,\cdots , \partial _{x_n}f_i):\PP^n \cdots \longrightarrow \PP^n$$ of $f_i.$ On the other hand, $E(T_1,\cdots ,T_r)$ is the closure of the union of points $x\in\PP^n$ such that $x$ lies in the linear space determined by $g^1(x), \cdots ,g^r(x)$.

 \begin{lemma}
 Fix integers $d_1, \cdots ,d_{n-1}\ge 2$ and $n-1$
  general partially symmetric tensors   $T_i=(g_0^{i},g_1^{i},\cdots ,g_n^{i})\in (Sym^{d_i-1}\kk$$^{n+1})^{\oplus n+1}$, $i=1, \cdots , n-1$. It holds:
 \begin{itemize}
     \item[(1)] $E(T_i)$ is a reduced 0-dimensional scheme in $\PP^n$ of length   $$
     \frac{(d_i-1)^{n+1}-1}{d_i-2}.
     $$
 \item[(2)]  The homogeneous ideal $I(E(T_i))\subset R$, $1\le i \le n-1$, has a minimal free $R$-resolution
     $$ 0 \longrightarrow \oplus _{j=0}^{n-1}R(-(j+1)d_i-n+2j+1) \longrightarrow   \cdots \longrightarrow $$
     $$R(-1-d_i)^{\binom {n+1} {3}}\oplus R(1-2d_i)^{\binom {n+1}{3}}
     \longrightarrow R(-d_i)^{\binom {n+1}{2}}\longrightarrow I(E(T_i)) \longrightarrow 0.
     $$
     \item[(3)]  $E(T_1,\cdots ,T_{n-1})$ is a reduced codimension 2 subscheme in $\PP^n$ of degree
     $$\frac{1}{2}(\sum _{s=1}^{n-1}(2-d_s -\sum _{j=1}^{n-1}d_j)
     )^2)+\frac{1}{2}(\sum _{j=1}^{n-1}d_j)^2 -\frac{n+1}{2}(-1+\sum _{j=1}^{n-1}d_j)^2.
     $$
     \item[(4)] The homogeneous ideal $I(E(T_1,\cdots ,T_{n-1}))\subset R$ has a minimal free $R$-resolution of the following type:
     $$ 0 \longrightarrow
  R(-\sum _{j=1}^{n-1}d_j)  \bigoplus \oplus _{s=1}^{n-1} R(2-d_s -\sum _{j=1}^{n-1}d_j)   )  \longrightarrow
     R(1-\sum _{j=1}^{n-1}d_j)^{n+1}$$
     $$
     \longrightarrow I(E(T_1,\cdots ,T_{n-1})) \longrightarrow 0.
     $$
     \end{itemize}
 \end{lemma}

 \begin{proof}
 Since $E(T_i)$, $i=1,\cdots , n-1$, and $E(T_1,\cdots ,T_{n-1})$ are standard determinantal scheme of codimension $n$ and $2$, respectively,  their minimal free resolution is given by the Eagon-Northcott complex (see \cite[Theorem A2.10]{e}) and the length of these schemes can be computed directly using the resolutions.
 \end{proof}

\begin{example}
\label{ex: Fermat}
\rm
We consider the Fermat cubic surface $f_1=x^3+y^3+z^3+t^3\in \kk$$[x,y,z,t]$ and the Clebsch cubic surface $f_2= x^2y+y^2z+z^2t+t^2x \in \kk$ $[x,y,z,t]$. The eigenscheme $E(f_1)$ is the 0-dimensional subscheme of $\PP^3$ of length 15 defined by the maximal minors of
$$
M_{f_1}=\begin{pmatrix} x & y & z & t\\
x^2 & y^2 & z^2 & t^2
\end{pmatrix}.
$$
Therefore,
$E(f_1)=\{(1,0,0,0),(0,1,0,0),(0,0,1,0),(0,0,0,1),(1,1,0,0),(1,0,1,0),(1,0,\\ 0,1),
(0,1,1,0),(0,1,0,1),(0,0,1,1),(1,1,1,0),(1,1,0,1),(1,0,1,1),(0,1,1,1),(1,1,1,1)\}.$

The eigenscheme $E(f_2)$ is the 0-dimensional subscheme of $\PP^3$ of length 15 defined by the maximal minors of
$$
M_{f_2}=\begin{pmatrix} x & y & z & t\\
2xy+t^2 & x^2+2yz &  y^2+2zt & z^2+2tx
\end{pmatrix}.
$$
Therefore, $E(f_2)=\{(1,0,0,0)\}.$

The ME-scheme $E(f_1,f_2)$ is the curve in $\PP^3$   defined by the maximal minors of
$$
M_{f_1,f_2}=\begin{pmatrix} x & y & z & t\\
x^2 & y^2& z^2& t^2 \\
2xy+t^2 & x^2+2yz &  y^2+2zt & z^2+2tx
\end{pmatrix}.
$$
Therefore, $$\begin{array}{l}I(E(f_1,f_2)=\\\langle x^2y^3-xy^4-x^4z+2xy^3z+x^3z^2-2x^2yz^2-2xy^2z^2+2xyz^3+2x^2yzt-2xy^2zt+y^2zt^2-yz^2t^2, \\ x^2yz^2-xy^2z^2-x^4t+2x^3yt-2x^2y^2t+2xy^3t-2x^2yzt+x^3t^2-2xy^2t^2+2xyzt^2+y^2t^3-yt^4,\\
x^2z^3-xz^4-x^2y^2t+2x^3zt-2x^2z^2t+2xyz^2t+xy^2t^2-2x^2zt^2-2xyzt^2+2xzt^3+z^2t^3-zt^4, \\ y^2z^3-yz^4-y^4t+2xy^2zt+x^2z^2t-2xyz^2t+2yz^3t+y^3t^2-x^2zt^2-2y^2zt^2-2yz^2t^2+2yzt^3
\rangle \end{array}$$
and  $E(f_1,f_2)$ is an arithmetically Cohen-Macaulay curve in $\PP^3$ of degree 17 and arithmetic genus 34 which contains the eigenschemes $E(f_1)$ and $E(f_2)$.
\end{example}

\section{Saito's theorem revisited}\label{Saito_sec}

Let us start this section with Saito's criterion as it was stated originally (See \cite{S}). Saito's criterion is a determinantal characterization of free divisors. Indeed, we have:

\begin{theorem} Let $S=V(f)\subset \PP^n$ be a reduced hypersurface of degree $d$. $S$ is free if and only if there exist $n$ derivations $\delta _i=g_0^{i}\partial _{x_0}+g_1^{i}\partial _{x_1}+\cdots + g_n^{i}\partial _{x_n}\in Der(f)$, $i=1,\cdots ,n$, such that
 $$ \det \begin{pmatrix} x_0 & x_1  & \cdots  & x_n \\
g_0^{1} & g_1^{1}  & \cdots  & g_n^{1}
\\
\vdots &  \vdots  & \cdots  & \vdots \\
g_0^n & g_1^n & \cdots  & g_n^n
\end{pmatrix} =cf \quad \text{ with }\quad  c\in \kk ^*
  $$
\end{theorem}

In \cite[Theorem 2.5]{DIMSV} we reformulate Saito's criterion in terms of eigenschemes and we get nice applications about the freeness of pencils of curves \cite[Theorem 3.6]{DIMSV}. In this section we generalize this result and we translate Saito's criterion in terms of ME-schemes. To this end the following technical lemma will be very useful:

\begin{lemma}\label{keylema}
 Fix integers $d_1, \cdots ,d_{n-1}\ge 2$ and let $T_i=(g_0^{i},g_1^{i},\cdots ,g_n^{i})\in (Sym^{d_i}\kk ^{n+1})^{\oplus n+1}$, $i=1, \cdots ,n-1$  be $n-1$
  partially symmetric tensors. Let $S=V(f)\subset \PP^n$ be a hypersurface of degree $d\ge d_1+\cdots +d_{n-1}-1$. $S$ contains $E(T_1,\cdots ,T_{n-1})$ if and only if there exists a $n$-th partially symmetric tensor $T=(g_0,g_1,\cdots ,g_n)\in (Sym^{d-d_1-\cdots -d_{n-1}+1}\kk^{n+1})^{\oplus n+1}$ such that
  $$ \det \begin{pmatrix} x_0 & x_1  & \cdots  & x_n \\
g_0^{1} & g_1^{1}  & \cdots  & g_n^{1}
\\ \vdots & \vdots & \cdots & \vdots \\
g_0^{n-1} & g_1^{n-1}  & \cdots  & g_n^{n-1} \\
g_0 & g_1 & \cdots  & g_n
\end{pmatrix} =cf \quad \text{ with }\quad  c\in \kk ^*
  $$
\end{lemma}
 \begin{proof} We consider the matrix $M_{T_1,\cdots ,T_{n-1}}$ associated to the tensors $T_1, \cdots ,T_{n-1}$. We denote by $h_i$ the determinant of the submatrix of $M_{T_1,\cdots ,T_{n-1}}$ obtained deleting the $i$-th column of  $M_{T_1,\cdots ,T_{n-1}}$. The hypersurface $S$ contains $E(T_1,\cdots ,T_{n-1})$ if and only if $(f)=I(S)\subset I(E(T_1,\cdots ,T_{n-1}))=\langle h_1,h_2,\cdots ,h_{n+1} \rangle $ if and only if there exist homogeneous polynomials $g_i\in R_{d-d_1-\cdots -d_{n-1}+1}$ such that $f=h_1g_0+h_2g_1+\cdots +h_{n+1}g_n$ or, equivalently, $$f=\det \begin{pmatrix} x_0 & x_1  & \cdots  & x_n \\
g_0^{1} & g_1^{1}  & \cdots  & g_n^{1}
\\ \vdots & \vdots & \cdots & \vdots \\
g_0^{n-1} & g_1^{n-1}  & \cdots  & g_n^{n-1} \\
g_0& g_1 & \cdots & g_n
\end{pmatrix}$$
 which proves what we want.
 \end{proof}

From now on given a non-zero irreducible derivation $\delta =P_0\partial _{x_0}+P_1\partial_{x_1}+\cdots +P_n\partial _{x_n}$ of degree $s$ we consider the graded  $R$-module of forms $f\in R$ with $\delta $ as a tangent derivation, i.e.
$$K(\delta):=\oplus _{r \ge 0}K(\delta)_r=\{f\in R \mid \delta(f)\in \langle f\rangle\};$$
and we call it the {\em kernel } of the derivation $\delta $. Notice that $F\in K(\delta)$ if and only if $\delta \in Der(F)$.

\begin{proposition}\label{saitorevisited} Let $S=V(f)\subset \PP^n$ be a reduced hypersurface of degree $d$. Let $\delta _i=g_0^{i}\partial _{x_0}+ g_1^{i}\partial _{x_1}+\cdots + g_n^{i}\partial _{x_n}$, $1\le i \le n-1$, be $n-1$ non-zero irreducible minimal generators of $Der(S)$ of degree $d_1\le \cdots \le  d_{n-1}$. Assume that  $E(T_1,\ldots,T_{n-1})$ has codimension $2$. Then, $S$ is free with exponents $(d_1,\cdots ,d_{n-1},d-d_1-\cdots -d_{n-1}-1)$ if and only if $S$ contains the ME-scheme $E(T_1,\cdots T_{n-1})$  associated to the partially symmetric tensors $T_i=(g_0^{i},g_1^{i},\cdots ,g_n^{i})\in (Sym^{d_i-1}\kk$$^{n+1})^{\oplus n+1}$, $i=1,\cdots ,n-1$.
\end{proposition}

\begin{proof}
Arguing as in \cite[Theorem 2.5]{DIMSV}, we start assuming that $S$ is free with exponents $(d_1,\ldots,d_{n-1},d-\Sigma_id_i-1)$, we can choose as base $\delta_1,\ldots,\delta_{n-1},\delta$ for suitable $\delta=g_0\partial_{x_0}+\cdots+g_n\partial_{x_n}$ of degree $d-\Sigma_i\delta_i-1$. Saito's criterion says us that 
\begin{equation}\label{eq_saito}
    \det \begin{pmatrix} x_0 & x_1  & \cdots  & x_n \\
g_0^{1} & g_1^{1}  & \cdots  & g_n^{1}
\\ \vdots & \vdots & \cdots & \vdots \\
g_0^{n-1} & g_1^{n-1}  & \cdots  & g_n^{n-1} \\
g_0& g_1 & \cdots & g_n
\end{pmatrix}=cf
\end{equation}
with $c\in \kk^*$ and this gives the thesis by Proposition \ref{saitorevisited}. Conversely, $S$ contains the $ME-$scheme if and only if there exists a tensor $T=(g_0,g_1,\cdots ,g_n)\in (Sym^{d-1}\kk$$^{n+1})^{\oplus n+1}$, such that \eqref{eq_saito} holds and we can also assume $c=1$. $T$ defines a derivation $\delta:=g_0\partial_{x_0}+\cdots+g_n\partial_{x_n}$. It is enough to prove that $\delta\in Der(S)$. If $\delta\notin Der(S)$, then there exists an irreducible factor. $g$ of $f$ that does not divide $\delta(f)$.
Arguing as in \cite{DIMSV}, let $M$ be the matrix in \eqref{eq_saito}
whose determinant is $cf$ and $CoM$ the matrix of cofactors
of $M$ and for any entry
$m_{ij}$,
let $m_{ij}^{'}$ be its cofactor.
We get that $g$ divides each generator of $I(E(T_1,\ldots,T_n))$, so the $ME$-scheme contains the hypersurface $V(g)$, this contradicts the codimension $2$.
\end{proof}

We are now ready to establish the revisited Saito's criterion for freeness.

\begin{theorem} \label{pencil}
Let $S_i: \ f_i=0$, $i=1,2$, be two reduced hypersurfaces of degree $d\ge 1$ in $\PP^n$ without common components. Assume that there exist $n-1$ partially symmetric tensors
$T_i=(g_0^{i},g_1^{i},\cdots ,g_n^{i})\in (Sym^{d_i-1}\kk$$^{n+1})^{\oplus n+1}$, $1\le i \le n-1$, such that
$\delta _i=g_0^{i}\partial _{x_0}+ g_1^{i}\partial _{x_1}+ \cdots  + g_n^{i}\partial _{x_n}$ are part of a minimal system of generators of both $Der_0(f_1)$ and $Der_0(f_2)$. If $S: \ f_1f_2=0$ is a free hypersurface with exponents $(d_1,\cdots ,d_{n-1},2d-d_1-\cdots -d_{n-1}-1)$ then $S_k: \ f_1f_2\prod _{j=1}^k(\alpha _jf_1+\beta _jf_2)=0$ is a free surface with exponents $(d_1,\cdots ,d_{n-1},(k+2)d-d_1-\cdots -d_{n-1}+1)$.
\end{theorem}

\begin{proof} Since $S: \ f_1f_2=0$ is free, by Proposition \ref{saitorevisited}, we know that $S$ contains the ME-scheme $E(T_1,\cdots ,T_{n-1})$ associated to the partially symmetric tensors $T_1, \cdots ,T_{n-1}$.

On the other hand,   $$\delta _i=g_0^{i}\partial _{x_0}+ g_1^{i}\partial _{x_1}+ \cdots  + g_n^{i}\partial _{x_n}
\in Der_0(f_1)\cap Der _0(f_2) \quad 1\le i \le n-1$$ Therefore, $\delta _i\in Der_0(f_1f_2\prod _{j=1}^k(\alpha _jf_1+\beta _jf_2))$ and, applying again Proposition \ref{saitorevisited}, we get that the surface $S_k$ is free with exponents $(d_1-1,\cdots ,d_{n-1}-1,(k+2)d-d_1-\cdots -d_{n-1}+n-2)$
if and only if $S_k$ contains $E(T_1,\cdots ,T_{n-1})$.
The result now follows from the observation that
 $E(T_1,,\cdots T_{n-1})\subset S=V(f_1f_2)\subset S_k=V(f_1f_2\prod _{j=1}^k(\alpha _jf_1+\beta _jf_2))$.
\end{proof}


\section{Applications}

Much of the early work on freeness focused on hyperplane arrangements and use as a key tool the so-called deletion-restriction operation \cite{T} or the composition formula \cite{BC} which, in particular, gives us the following result: if $f\in \kk[x_0,\cdots ,x_n]$ defines a free hypersurface in $\PP^{n}$ and $g\in \kk[y_0,\cdots ,y_m]$ defines a free hypersurface in $\PP^{m}$ then $fg$ and $fg(f+g)$ define free hypersurfaces in $\PP^{n+m+1}$.

In this section, we will use the
geometry of arithmetically Cohen-Macaulay hypersurfaces naturally associated to partially symmetric tensors to gain some knowledge about the freeness of pencils of hypersurfaces. In fact, as application of the Saito's revisited freeness criterion (see Theorem \ref{pencil}) we will be able to construct free hypersurfaces in $\PP^n$, $n\ge 3$, of arbitrarily high degree. Let us start with a toy example

\begin{example}
    \label{cubicspencil} In this example, we will construct free surfaces in $\PP^3$ all of them are union of surfaces in a suitable pencil of cubics.
    Take $f_1=xyz+xzt$ and $f_2=xyt+yzt$. The pencil of cubic surfaces $S_k: \ f_1f_2\prod _{j=1 }^k (a_jf_1+b_jf_2)=0$  in $\PP^3$ is free with exponents $(2,2,3k+1).$ Indeed, $Q_1=(2x^2+xz,-xy+yz,-2z^2-xz,-xt+ zt)$ and $Q_2=(xy-xt,-2y^2-yt,yz-zt,yt+2t^2)$ are two minimal syzygies of degree 2 of the Jacobian ideal of $f_1$ as well as the Jacobian ideal of $f_2$. Moreover, $f_1f_2$ is a free surface in $\PP^3$ with exponents $(1,2,2).$ Therefore, applying Theorem \ref{pencil} we conclude that $f_1f_2\prod _{k=3}^r(a_kf_1+b_kf_2)$
with $a_k,b_k\in \kk\setminus \{0\}$ is free with exponents $(2,2,1+3(k-2))$.
\end{example}

Our goal is to  generalize this last  example. To this end, let us start fixing some notation. Write $n=2m+\epsilon $ with $\epsilon =0,1$.  We consider the hyperplane arrangement $\cA: \ x_0\cdot ... \cdot x_n=0$ of the $n+1$ coordinate hyperplanes in $\PP^n$. $\cA $ is a free arrangement with exponents $(1,\cdots ,1)$, i.e.,
$$
\cT_{\cA} \simeq \cO_{\PP^n}(-1)^n.
$$
The Jacobian ideal of $\cA$ is
$$
J=\langle x_1x_2\cdots x_n, x_0x_2\cdots x_n, \cdots, x_0x_1\cdots x_{n-1} \rangle
$$
and the singular locus consists of the ${n+1\choose 2}$ codimension two faces of the hypertetrahedron. Set 
$$h_i=x_0\cdot ...\cdot \widehat{x_i}\cdot ... \cdot x_n, \ i=0,\cdots ,n.$$ We consider the hypersurfaces  $S_1$ and $S_2$ of degree $n$  defined by  $$f_1=\sum _{i=m+1}^n h_i \text{ and } f_2=\sum _{i=0}^mh_i$$ and we denote by $$\cP :=\{S_{a,b}=aS_1+bS_2\}_{ a,b\in \kk}$$ the pencil of hypersurfaces of degree $n$ determined by $S_1$ and $S_2$.

\begin{theorem} \label{mainthm} With the above notation, we have
\begin{itemize}
   \item[(i)] $S_1S_2$ is free with exponents $(1,2,\cdots ,2)$
   \item[(ii)] For any $k\ge 3$ and for generic hypersurfaces $S_{a_i,b_i}$, $3\le i \le k$, in the pencil $\cP$, the hypersurface 
$$S_1S_2\prod _{i=3}^k S_{a_i,b_i}$$
is free with exponents $(2,...,2,n(k-2)+1)$.
\end{itemize}
\end{theorem}

\begin{proof}We will assume that $n=2m$; analogous argument works for $n=2m+1$ as the reader can check. We consider the $m$ partial symmetric tensors $R_i=(R_i^0,R_i^1,\cdots ,R_i^n)\in (Sym^2\kk^{m+1})^{\oplus n+1}$ with $1\le i\le m$ and the $m-1$  partial symmetric tensors $S_j=(S_j^0,S_j^1,\cdots ,S_j^n)\in (Sym^2\kk^{m+1})^{\oplus n+1}$ with $1\le j\le m-1$ defined as follows:
$$R_i=(x_0(x_{i-1}-x_m),\cdots , x_{i-1}(-(n-1)x_{i-1}-x_m),\cdots , x_m(x_{i-1}+(n-1)x_m),$$
$$x_{m+1}(x_{i-1}-x_m),\cdots ,x_n(x_{i-1}-x_m)) \text{ for } i=1,\cdots , m;$$
and
$$S_j=(x_0(x_{m+j}-x_n), \cdots ,x_m(x_{m+j}-x_n),$$
$$ x_{m+1}(x_{m+j}-x_n),\cdots ,x_{m+j}(-(n-1)x_{m+j}-x_n),\cdots ,x_n(+(n-1)x_n+x_{m+j})) \text{ for } j=1,\cdots , m-1. $$

We easily check that for any integers $0\le t\le n$, $1\le i \le m$ and $0\le j \le m-1$ it holds:

$$\delta _i=\sum _{s=0}^nR_i^s\partial_{x_s}\in Der _0(f_1)\cap Der_0(f_2)$$
and
$$\eta _j=\sum _{s=0}^nS_j^s\partial_{x_s}\in \bigcap_{t\leq m} Der _0(h_t) \cap Der_0(f_1)\subset Der_0(f_1)\cap Der_0(f_2)$$
(i) We first observe that $f_1$ and $f_2$ are two reduced  hypersurfaces of degree $n$  without common factors and $\deg(f_1f_2)=2n\ge \sum _{i=1}^m\deg R_i+\sum _{j=1}^{m-1}\deg S_j=2n-2$. Moreover, a straightforward computation shows that $\delta _i $ and $\eta _j$ are part of a minimal system of generators of $Der_0(f_1f_2)$ and hence we conclude applying Theorem \ref{pencil}.

(ii) 
Arguing as above we get again  that  $\delta _i $ and $\eta _j$ are part of a minimal system of generators of $Der_0(f_1f_2\prod _{i=1}^k (a_if_1+b_if_2))$. Therefore,we can apply  Theorem \ref{pencil} we get what we want.
\end{proof}
\begin{example}\label{lastex} Now we consider again the arrangement $\cA$ of the 4 coordinate planes in $\PP^3$. As already noted, it is free, with 3 linear syzygies, hence
$$
\cT_A \simeq \cO_{\PP^3}(-1)^3.
$$
The Jacobian ideal of $\cA$ is obviously
$$
\langle xyz, xyw,xzw, yzw \rangle
$$
and the corresponding singular locus consists of the six edges of tetrahedron. 

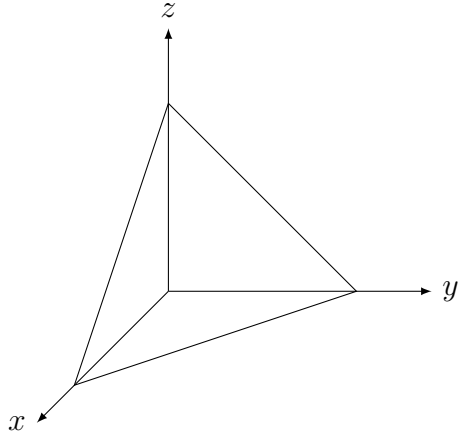
\begin{center}
\begin{tikzpicture}[y ={(1 cm ,0 cm )} , x ={(-0.5 cm ,-0.5 cm )} ,z ={(0 cm ,1 cm )}] 
 \coordinate (O) at (0,0,0);
 \draw [-latex ] 
 (O) -- (3.5,0,0) node [ left ] {$ x $}
 ;
\draw 
[-latex ] 
(O) -- (0,3.5,0) node [ right ] {$ y $}
;
  \draw [-latex ] (O) -- (0,0,3.5) node [ above ] {$ z $};
\draw (2.5 ,0 ,0) -- (0 ,2.5 ,0) -- (0 ,0 ,2.5) -- cycle ;
 \end{tikzpicture}    
\end{center}

Picking a generic cubic $C_1$ in this web and adding it to $\cA$ we obtain an arrangement $\cA \cup C_1$. A computation shows it has quasi-homogeneous singularities, and
$$\cT_{\cA \cup C_1} \simeq \cO(-2)^3.$$
If we add a second generic cubic $C_2$ from the web to this arrangement, we obtain an arrangement $\cA \cup C_1 \cup C_2$ which is now singular along a complete intersection curve of degree 9, consisting of the 6 singular lines of $\cA$, and a twisted cubic; the fact that a twisted cubic is a component follows using liaison.
Now there are two directions to go. If we continue to add cubics from the pencil of cubics defined by $C_1$ and $C_2$, and call $\mathcal C_k$ the resulting arrangement of the four coordinate planes and the $k$ cubics $C_1,\ldots,C_k$, then we compute that for all $k$,
$$
\cT_{\mathcal C_k} \simeq \cO(-2)^2 \oplus \cO(-3k+1).
$$

On the other hand, we can continue to add cubics from the original web. Denote such an arrangement, where $k$ cubics from the web have been added, as $\mathcal{C}'_k$. We find that
$$
  \begin{array}{ccc}
    \cT_{{\mathcal C}'_1} &\simeq &\cO(-2)^3.\\
    \cT_{{\mathcal C}'_2} &\simeq & \cO(-2)^2 \oplus \cO(-5).\\
    \cT_{{\mathcal C}'_3} &\simeq & \cO(-2) \oplus \cO(-5)^2.\\
    \cT_{{\mathcal C}'_4} &\simeq & \cO(-5)^3.\\
    \cT_{{\mathcal C}'_5} &\simeq & \cO(-6)^3.\\
    \cT_{{\mathcal C}'_k} &\mbox{ is not free if } &k \ge 6.\\
  \end{array}
$$

The above example generalizes to $\PP^n$. Indeed, we consider the arrangement $\cA$ of the $n+1$ coordinate hyperplanes in $\PP^n$. $\cA $ is  free with exponents $(1,\cdots ,1)$, i.e.,
$$
\cT_{\cA} \simeq \cO_{\PP^n}(-1)^n.
$$
The Jacobian ideal of $\cA$ is
$$
J=\langle x_1x_2\cdots x_n, x_0x_2\cdots x_n, \cdots, x_0x_1\cdots x_{n-1} \rangle
$$
and the singular locus consists of the ${n+1\choose 2}$ codimension two faces of the hypertetrahedron. Let $S_1$ be a hypersurface of degree $n$ defined by a general form $f_1\in J$  and add it to $\cA$. We get a free arrangement $\cA \cup C_1$ with
$$
\cT_{\cA\cup S_1} \simeq \cO_{\PP^n}(-2)^n.
$$
Add a second hypersurface $S_2$ of degree $n$ defined by a general form $f_2\in J$  and add it to $\cA$. We have
$$
\cT_{\cA\cup S_1\cup S_2} \simeq \cO_{\PP^n}(-2)^{n-1}\oplus \cO_{\PP^n}(-2-n).
$$
 More general, we add hypersurfaces from the pencil defined by $S_1$ and $S_2$. The arrangement
 ${\mathcal C}_k:$ $x_0\cdots x_n\prod _{i=1}^k (\alpha _if_1+\beta _if_2)$
 is free with exponents $(2,\cdots ,2,2+n(k-1)$. Therefore, we have:
 $$
\cT_{{\mathcal C}_k} \simeq \cO_{\PP^n}(-2)^{n-1}\oplus \cO_{\PP^n}(-2-n(k-1)).
$$
\end{example}

\begin{remark}
As highlighted in \cite[Theorem 5.2]{DIMSV} and Example \ref{lastex}, if, instead of adding hypersurfaces from the pencil $\cP$, we add hypersurfaces from the linear system $J$ the result is no longer true when the number of hypersurfaces is at least 5. These are only examples, but we computed, by Macaulay2, several situations of this type in several dimensions and the behavior is always the same. We are investigating them theoretically. 
\end{remark}



\begin{thebibliography}{ll}

\bibitem{Abo} {H. Abo},
	{\em On the discriminant locus of a rank $n-1$ vector bundle on
	$\PP^{n-1}$},
	{Portugaliae Mathematica}
	{\bf 77}
	(2020), {299--343}.

\bibitem{BC} R.O. Buchweitz and A. Conca, {\em New free divisors from old}, Journal of Commutative Algebra {\bf 5} (2013), {17--47}.


     \bibitem{DIMSV} R. Di Gennaro, G. Ilardi, R.M. Mir\'o-Roig, H. Schenck and J. Vall\`es, {\em Free curves, eigenschemes and pencils of curves}, Bull. London Math. Soc. {\bf 56} (2024), {2424--2440}.

\bibitem{DS15} A. Dimca and G. Sticlaru, {\em Free and nearly free surfaces in $P^3$}, Asian J. Math {\bf 22} (2018), {787--810}.

\bibitem{DS17} A. Dimca and G. Sticlaru, {\em Computing Milnor fiber monodromy for some projective hypersurfaces}, Contemp. Math. {\bf 742} (2020), {31--52}.



\bibitem{e} D. Eisenbud,
        {\em Commutative Algebra with a view towards Algebraic Geometry},
        Graduate Texts in Mathematics, vol.~150,
        Springer-Verlag, Berlin-Heidelberg-New York, 1995.

\bibitem{HE}
    M. Hochster, J. A. Eagon, {\em Cohen-{M}acaulay rings, invariant theory, and the generic
              perfection of determinantal loci},
  {Amer. J. Math.},
  {\bf 93},
      (1971),
     {1020--1058}.

 \bibitem{L}  L.H. Lim. {\em Singular values and eigenvalues of tensors: a variational approach.} In 1st IEEE International Workshop
on Computational Advances in Multi-Sensor Adaptive Processing (2005), {129--132}.

\bibitem{Q}  L. Qi. {\em Eigenvalues of a real supersymmetric tensor}, Journal of Symbolic Computation {\bf 40} (2005), {1302--1324}.

\bibitem{S} K. Saito, {\em Theory of logarithmic differential forms and logarithmic
vector fields}, J. Fac. Sci. Univ. Tokyo {\bf 27} (1980), {265--291}.

\bibitem{ST} H. Schenck, S.O. Tohǎneanu, {\em Freeness of conic-line arrangements in $\PP^2$}, Comment. Math. Helv. {\bf 84} (2009), {235--258}.
\bibitem{T}  H. Terao,  {\em Arrangements of hyperplanes and their freeness I, II}, . J. Fac. Sci. Univ. Tokyo {\bf  27} (1980), {293--320}.


\bibitem{V} J. Vall\` es, {\em Free divisors in a pencil of curves}, Journal of Singularities \textbf{11} (2015), {190--197}.

\end{thebibliography}
\end{document}